\documentclass{amsart}
\usepackage{amscd}
\usepackage{amssymb}
\usepackage[all, knot]{xy}
\xyoption{all}
\xyoption{arc}
\usepackage{hyperref}

\def\op{\text{op}}

\def\Z{\mathbb{Z}}
\def\cO{\mathcal{O}}

\def\Id{\mathrm{Id}}

\def\coh{\operatorname{coh}}
\def\qcoh{\operatorname{Qcoh}}

\def\Spec{\operatorname{Spec}}
\def\Proj{\operatorname{Proj}}

\numberwithin{equation}{section}

\theoremstyle{plain} 
\newtheorem{thm}[equation]{Theorem}

\newtheorem*{introthm*}{Theorem}

\newtheorem{cor}[equation]{Corollary}
\newtheorem{lem}[equation]{Lemma}
\newtheorem{prop}[equation]{Proposition}

\theoremstyle{definition}
\newtheorem{defn}[equation]{Definition}

\newtheorem{ex}[equation]{Example}

\theoremstyle{remark}
\newtheorem{rem}[equation]{Remark}

\newtheorem*{rems}{Remark}

\newtheorem*{ack}{Acknowledgements}

\newenvironment{claim}[1]{\par\noindent\emph{Claim}\space#1:}{}
\newenvironment{claimproof}{\par\noindent\emph{Proof of claim:}}{\hfill}

\def\coh{\operatorname{coh}}
\def\coker{\operatorname{coker}}

\DeclareMathOperator*{\colim}{colim}

\newcommand{\Ext}[4]{\operatorname{Ext}_{#2}^{#1}(#3,#4)}
\newcommand{\uExt}[4]{\operatorname{\underline{Ext}}_{#2}^{#1}(#3,#4)}

\newcommand{\Hom}[3]{\operatorname{Hom}_{#1}(#2,#3)}

\newcommand{\uHom}[3]{\operatorname{\underline{Hom}}_{#1} (#2, #3) }
\newcommand{\RuHom}[3]{\R\operatorname{\underline{Hom}}_{#1} (#2, #3) }

\newcommand{\li}{ < \infty}

\newcommand{\xra}[1]{\xrightarrow{#1}}

\newcommand{\psc}[1]{\mathbb{P}_{#1}^{\text{c}-1}}

\newcommand{\id}{\operatorname{id}}

\newcommand{\grA}{\operatorname{gr}A}
\newcommand{\grAop}{\operatorname{gr}A^\op}
\newcommand{\GrA}{\operatorname{Gr}A}

\newcommand{\grAi}{\operatorname{gr}_{\geq i}A}
\newcommand{\grAo}{\grA_{\geq 0}}

\newcommand{\GrAi}{\operatorname{Gr\,}_{\geq i}A}
\newcommand{\Gi}{\Gamma_{\geq i}}
\newcommand{\torsA}{\operatorname{tors}A}
\newcommand{\torsAi}{\operatorname{tors}_{\geq i}A}
\newcommand{\TorsA}{\operatorname{Tors}A}
\newcommand{\TorsAi}{\operatorname{Tors}_{\geq i}A}

\newcommand{\Dp}[1]{\mathsf{D}( #1 )}
\newcommand{\Db}[1]{\mathsf{D}^{\mathsf{b}}( #1 )}
\newcommand{\Dsg}[1]{\mathsf{D}^{\mathsf{b}}_{\mathsf{sg}}( #1 )}
\newcommand{\T}{\mathcal{T}}

\newcommand{\D}{\mathcal{D}}

\newcommand{\gend}[1]{\,\langle ~\!\! #1 \!\!\!~\rangle}
\newcommand{\Sg}[1]{\mathcal{S}_{\geq #1}}
\newcommand{\Sl}[1]{\mathcal{S}_{< #1}}
\newcommand{\Pg}[1]{\mathcal{P}_{\geq #1}}
\newcommand{\Pl}[1]{\mathcal{P}_{< #1}}
\newcommand{\Bi}{\mathcal{B}_i}

\newcommand{\X}{X}
\newcommand{\twid}[1]{\widetilde{#1}}

\newcommand{\R}{\boldsymbol{\operatorname{R}}}

\newcommand{\perf}[1]{\operatorname{perf}(#1)}

\newcommand{\bb}{\textbf{b}}

\newcommand{\tors}[1]{\tau(#1)}

\newcommand{\RY}{\beta}

\newcommand{\bE}{\mathbb{E}}
\newcommand{\emf}[1]{(E_1 \xra{e_1} E_0
\xra{e_0} E_1(#1))}
\newcommand{\fmf}[1]{(F_1 \xra{f_1} F_0 \xra{f_0} F_1(#1))}

\newcommand{\grmf}{\operatorname{gr-mf}(S, W)}
\newcommand{\hgrmf}{[\grmf]}
\newcommand{\grS}{\operatorname{gr}S}
\renewcommand{\bf}{\boldsymbol{f}}

\begin{document}
\title[The derived category of a graded Gorenstein ring]{The derived category of a graded \\Gorenstein ring}
\author{Jesse Burke}
\address{Mathematics Department\\ 
UCLA\\ 
Los Angeles, CA\\ 90095-1555, USA}
\email{jburke@math.ucla.edu}
\author{Greg Stevenson}
\address{Universitat Bielefeld, Fakult\"{a}t f\"{u}r Mathematik, BIREP Gruppe, Postfach 10 01 31, 33501
Bielefeld, Germany.}
\email{gstevens@math.uni-bielefeld.de}

\begin{abstract} 
We give an exposition and generalization of Orlov's theorem on graded Gorenstein
rings. We show the theorem holds for non-negatively graded rings which are Gorenstein in an appropriate sense and whose degree zero
component is an arbitrary non-commutative right noetherian ring of finite
global dimension. A short treatment of some foundations for local cohomology and
Grothendieck duality at this level of generality is given in order to prove the
theorem. As an application we give an equivalence of the
derived category of a commutative complete intersection with the homotopy category of graded
matrix factorizations over a related ring.
\end{abstract}

\maketitle

\section{Introduction}
Let $A$ be a graded Gorenstein ring. In \cite{MR2641200} Orlov related the bounded derived
category of coherent sheaves on $\Proj A$ and the singularity category
of graded $A$-modules via fully faithful functors; the exact relation depends on the $a$-invariant
of $A$. This is a striking theorem that has found applications in
physics, algebraic geometry and representation theory. To give an idea
of the scope of the theorem: in the limiting case that $A$ has finite
global dimension (so the singularity category is trivial), it recovers (and generalizes to non-commutative
rings) Beilinson's result \cite{MR509388} that the derived category of $\Proj A$ is
generated by a finite sequence of twists of the structure sheaf.

There has been much work related to this theorem. The idea of Orlov's construction perhaps first appears in Van Den
Bergh's paper on non-commutative crepant resolutions
\cite{MR2077594} where he described functors similar to those
considered by Orlov, in the case of torus invariants. After Orlov's paper appeared, the
idea was further explored by the physicists Herbst, Hori and Page in
\cite{0803.2045}. In turn these ideas were the inspiration for two
papers on the derived category of GIT quotients \cite{1203.6643,1203.0276}. Segal and
then Shipman gave geometric proofs of Orlov's theorem for commutative
hypersurfaces in \cite{MR2795327} and \cite{MR2982435}. Related
results are \cite{MR2593258} and \cite{1011.1484}. The theorem has
been used in similar ways in \cite{MR2947547,MR2776613}. Finally, it
has been used in representation theory; especially in the study of
weighted projective lines, see e.g.\ \cite{MR2931898}.

Orlov assumed that $A_0$ a field. In this paper, which we consider
largely expository,
we generalize his result to show that the same relation holds when
$A_0$ is a non-commutative
noetherian ring of finite global dimension. This has an immediate
application to commutative complete intersection rings and we expect
there to be further applications, for instance to (higher) preprojective
algebras. The structure of our proof is very close to Orlov's original arguments. We give many details and we
hope that these details may help the reader (even one only interested in
algebras defined over a field) to better understand Orlov's work.

The main tool in the proof is a semiorthogonal
decomposition; this separates a triangulated category into an
admissible subcategory and its orthogonal. Derived global sections
gives an embedding of $\Db {\Proj
  A}$ into $\Db {\grAi}$ as an admissible subcategory. When $A_0$ is a
field, Orlov showed that there is an
embedding of the singularity category $\Dsg \grA$ into $\Db \grAi$. The
existence of such an embedding is rather remarkable and constitutes
perhaps the key insight required to prove the theorem. Orlov then used
Grothendieck duality in a
very clever way to compare the orthogonals of $\Db {\Proj
  A}$ and $\Dsg \grA$ inside of $\Db \grAi$. For example when $A$ is
Calabi-Yau, the orthogonals coincide and so there is an equivalence
between $\Db {\Proj
  A}$ and $\Dsg \grA$.

The arguments we present here follow those of Orlov. The main addition
is the observation that, when $A_0$ has finite global dimension, one
can construct particularly nice resolutions of complexes of graded
modules with bounded finitely generated cohomology. This allows us to
prove Orlov's embedding $\Dsg \grA \to \Db \grAi$ is valid in this
more general setting. We also need to develop some foundations
concerning local cohomology and Grothendieck duality over
non-commutative rings to prove analogues of the other steps of Orlov's proof.
These foundations do not seem
to be contained in the literature in the form and generality that we
need, although the arguments we give here are relatively straightforward
generalizations of arguments by Artin and Zhang \cite{MR1304753}.


Let us give a quick summary of the paper. The second section contains
some categorical background, especially on semiorthogonal
decompositions. The third section is devoted to the derived category
of graded modules over a
graded ring, and some standard semiorthogonal decompositions that
appear there. This section contains the key observation, Lemma~\ref{gregs_lemma}, needed to prove the embedding of the singularity category works for the rings we work with. The fourth section deals with local cohomology and the
semiorthogonal decomposition it gives, while the fifth deals with the
embedding of the singularity category, Grothendieck duality, and the
semiorthogonal decomposition these give. The sixth section contains
the proof of the main result as well as a sufficient condition for a Gorenstein ring to satisfy Artin and Zhang's condition
$\chi$, which is necessary for the proof. Finally, in the last section, we apply
the main theorem to give a description of the bounded derived
category of a complete intersection ring in terms of graded matrix factorizations.

\begin{ack}
Ragnar-Olaf Buchweitz gave two
beautiful lectures on Orlov's theorem at Bielefeld University in
July 2011 that inspired this and from which we learned a lot. The existence of the
crucial left adjoint to the projection onto the singularity category
over general bases of finite global dimension is implicitly
contained in
Ragnar's exposition of Orlov's work. We thank Mark Walker for
several helpful conversations on the contents of the paper.
\end{ack}

\section{Background}
We recall here some standard results on semiorthogonal decompositions of 
triangulated categories which we will need. Throughout this section $\T$ denotes a triangulated category.

\begin{defn}
For $\D$ a triangulated subcategory of $\T$, define $\D^\perp$ to be the full
subcategory with objects those $X \in \T$ such that $\Hom \T D X = 0$
for all objects $D$ of $\D$. Similarly, ${}^\perp\D$ has objects those
$X$ with $\Hom \T X D = 0$ for all $D$ in $\D$. Both $\D^\perp$ and ${}^\perp\D$ 
are thick subcategories of $\T$ i.e., they are triangulated subcategories which 
are closed under taking direct summands.
\end{defn}

\begin{defn}
A triangulated subcategory $\D$ of $\T$ is \emph{left admissible} in $\T$ if
the inclusion functor $i: \D \to \T$ has a left adjoint; $\D$ is
\emph{right admissible} if $i$ has a right adjoint.
\end{defn}

The following criterion for admissibility can be found as \cite[Lemma~3.1]{MR992977}.

\begin{lem}
\label{admiss_crit}
 Let $\D$ be a triangulated subcategory of $\T$. 
 \begin{enumerate}
\item The category $\D$ is left admissible if
 and only if for every $X$ in $\T$ there is a triangle:
\[ E_X \to X \to D_X \to \Sigma E_X\]
with $D_X$ in $\D$ and $E_X$ in $^\perp D$.
 \item The category $\D$ is right admissible if and only if  for every $X$ in $\T$ there is a triangle:
   \[ D_X \to X \to E_X \to \Sigma D_X\] with ${D_X}$ in $\D$ and
   $E_X$ in $\D^\perp$.
 \end{enumerate}
\end{lem}

\begin{cor}
 A subcategory $\D$ of $\T$ is left admissible if and only if ${}^\perp
\D$ is right admissible. In this case $\left ( {}^\perp\D \right )^\perp =
   \D$.
\end{cor}




\begin{defn}
A \emph{semiorthogonal decomposition} of
$\T$ is a pair of subcategories $\mathcal{A}$ and $\mathcal{B}$ such that
$\mathcal{A}$ is left admissible and $\mathcal{B} = {}^\perp
\mathcal{A}$ (equivalently, $\mathcal{B}$ is right admissible and
$\mathcal{A} = \mathcal{B}^\perp$). We write this as
\[ \T = \left ( \mathcal{A}, \mathcal{B} \right ).\]
\end{defn}

The following useful lemma essentially follows from Lemma~\ref{admiss_crit} in a straightforward way.

\begin{lem}
\label{crit-for-so-decomp}
 There is a semiorthogonal decomposition $\T =  \left ( \mathcal{A},
  \mathcal{B} \right )$ if and only if $\mathcal{B} \subseteq {}^\perp
\mathcal{A}$ and for every $X$ in $\T$ there is a triangle
\[ B_X \to X \to A_X \to \]
with $B_X \in \mathcal{B}$ and $A_X \in \mathcal{A}$.
We will call such a triangle the \emph{localization triangle} for
$X$.
\end{lem}

Orlov generalized the definition of semiorthogonal decomposition in
\cite{MR2641200} to:
\begin{defn}
  A sequence of full triangulated subcategories
  $(\D_1, \ldots, \D_n)$ of $\T$ is a \emph{semiorthogonal decomposition}
  if for each $i = 1, \ldots, n - 1$, the thick subcategory generated by $\D_1,
  \ldots, \D_i$, which we denote $\gend
  {\D_1, \ldots, \D_i}$, 
is left admissible and 
\[^\perp \gend {\D_1, \ldots, \D_i} = \gend{
  \D_{i+1}, \ldots, \D_{n}}.\]
\end{defn}

We can (and will) construct semi-orthogonal decompositions inductively:

\begin{lem}
\label{build-up-of-so-decomps}
 Let $\T = (\mathcal{A}, \mathcal{B})$, $\mathcal{A} = (\D_1, \ldots, \D_i),$ and
 $\mathcal{B} = (\D_{i+1}, \ldots, \D_n)$ be semiorthogonal
 decompositions. Then
\[ \T = (\D_1, \ldots, \D_n)\]
is a semiorthogonal decomposition.
\end{lem}




\section{The bounded derived category of graded modules}
We now provide some preliminary results on derived categories of graded modules. We 
begin by exhibiting some semiorthogonal decompositions of the bounded derived 
category which we will need in the sequel (and which are relatively 
straightforward generalisations of those in Orlov's work). We also prove the 
main technical results concerning graded projectives and graded projective 
resolutions which we will need for our extension of Orlov's result.

In this section $A = \bigoplus_{i \geq 0} A_i$ is a positively graded
 noetherian ring with $A_0$ a ring of finite global dimension.
All modules
 will be right modules unless otherwise stated. We denote by $\grA$ the abelian category of finitely generated graded
$A$-modules and degree zero homogeneous maps. If $M = \bigoplus M_i$ is a graded
$A$-module, then $M(1)$ is the graded $A$-module with $M(1)_i =
M_{i+1}$.


We denote by
$\grAi$ the full subcategory of $\grA$ consisting of objects $M$ such
that $M_j = 0$ for all $j < i$. This is an abelian subcategory of
$\grA$ and there is an adjoint pair of functors
\begin{displaymath}
\xymatrix{ \grAi \ar@<1ex>[r]^{\operatorname{inc}} & \grA
  \ar@<1ex>[l]^{(-)_{\geq i}}}
\end{displaymath}
where $M_{\geq i} = \bigoplus_{j \geq i} M_j$ is right adjoint to the inclusion.

We denote by $\Db -$ the bounded derived category of an abelian
category. Both functors of the above adjoint pair are exact and so
induce functors 
\begin{displaymath}
\xymatrix{ \Db \grAi \ar@<1ex>[r]^{\operatorname{inc}} & \Db \grA
  \ar@<1ex>[l]^{(-)_{\geq i}}}
\end{displaymath}
which also form an adjoint pair. The functor induced by inclusion is
fully faithful and the essential image is
the full subcategory of $\Db \grA$ consisting of objects $M$ such that
$H^j(M) \in \grAi$ for all $j\in \Z$. We denote this subcategory also
by $\Db \grAi$; it is a right admissible subcategory.


\begin{defn}
 Define $\Sl i$ to be the thick subcategory generated by the
 objects $A_0(e)$, for all $e > -i$ and $\Sg i$ to be the
 thick subcategory generated by $A_0(e)$, for all $e \leq -i$. 
\end{defn}

\begin{lem}
\label{desc_Sl}
 An object $M$ of $\Db \grA$ is in $\Sl i$ if and only if $M_{\geq i}
 \simeq 0$.
\end{lem}

\begin{proof}
The full subcategory with objects those $M$ satisfying $M_{\geq i} \simeq 0$
is thick by virtue of being the kernel of an exact functor. 
Since $A_0(e)_{\geq i} = 0$ for all $e > -i$, we
see that $\Sl i$ is contained in this thick subcategory. Thus if $M$
is in $\Sl i$ we must have $M_{\geq i} \simeq 0$.

For the converse, we first assume $M$ is a module, i.e.\
concentrated in homological degree $0$, and that there is an
integer $e < i$ with $M_j = 0$ for all $j \neq e$. Since $M_e$ is a
finitely generated $A_0$-module and $A_0$ has finite global dimension, $M_e$ has
a finite projective resolution over $A_0$. Over $A$ this says that $M_e$ is
in the thick subcategory generated by $A_0(-e)$, which is contained in $\Sl i$.

Now suppose $M$ is a non-zero
finitely generated graded $A$-module with $M_{\geq i}\simeq 0$. As $M$ is finitely generated there is an integer $j$ with $M_{\geq
  j} = M$ and we may as well choose a maximal such $j$, which is necessarily less than $i$.  Consider the triangle
\[ M_{\geq j +1} \to M \to M_j \to \]
By the previous argument $M_j$ is in $\Sl i$ and arguing inductively
on the number of degrees in which $M$ is non-zero we see that $M_{\geq j + 1}$ is in $\Sl i$, and hence $M$ is in
$\Sl i$.

For an arbitrary non-zero object $M\in \Db\grA$, with $M_{\geq i}
\simeq 0$ the result follows from induction on the number of
non-vanishing cohomology modules, using the triangle
\[ M^{<n} \to M \to H^n(M)[-n] \to\]
where $n = \max \{ j \, | \, H^j(M) \neq 0 \}$ and $M^{<n}$ is the truncation with respect to the standard t-structure.
\end{proof}

\begin{rem}
 It follows from the definition that $\Sg i$ is contained in $\Db
 \grAi$. We will show in Lemma \ref{sgi-tors-lemma} that $\Sg i$ is the full subcategory of
 $\Db\grAi$ whose objects have torsion cohomology.
\end{rem}

\begin{lem}
\label{sli_admiss}
 There is a semiorthogonal decomposition
\[\Db
 \grA = \left (\Sl i, \Db \grAi \right ).\] The localization triangle
 for $M \in \Db \grA$ is given by the canonical maps
\[ M_{\geq i} \to M \to M/M_{\geq i} \to. \]
\end{lem}

\begin{proof}
Let $M$ be in $\Sl i$ and $N$ be in $\Db \grAi$. Then we have that
\[ \Hom {\Db \grA} N M \cong \Hom {\Db \grAi} N {M_{\geq i}} = 0\]
by right adjointness of $(-)_{\geq i}$ and since $M_{\geq i} \simeq
0$. Thus $\Db \grAi \subseteq {}^\perp
{\Sl i}$. If $M$ is any object in $\Db \grA$ we have the triangle
\[ M_{\geq i} \to M \to M/M_{\geq i} \to \]
with $M_{\geq i}$ in $\Db \grAi$ and $M/M_{\geq i}$ in $\Sl i$. Thus
we may apply
Lemma \ref{crit-for-so-decomp}.
\end{proof}

\begin{defn}
 Define $\Pl i$ to be the thick subcategory generated by the
 objects $A(e)$ for all $e > -i$ and $\Pg i$ to the thick
 category generated by $A(e)$ for all $e \leq -i$.
\end{defn}

\begin{rem}
 It follows from the definition that $\Pg i$ is contained in $\Db
 \grAi$. In fact, $\Pg i$ is the full subcategory of $\Db\grAi$ whose
 objects are perfect complexes of $A$-modules.
\end{rem}

\begin{lem}
\label{Pl_admiss}
There is a
 semiorthogonal decomposition 
\[ \Db \grA = (\Db \grAi, \Pl i).\]
\end{lem}

Before proving this lemma, we need two results on graded projective
$A$-modules and graded projective resolutions over $A$.

\begin{lem}
\label{lem_projdesc}
Let $P$ be a finitely generated graded projective $A$-module. Then
there is an isomorphism, for some integers $n, m_1, \ldots, m_n$,
\begin{displaymath}
P \cong \bigoplus_{i=1}^n P_i \otimes_{A_0} A(m_i)
\end{displaymath}
where the $P_i$ are projective right $A_0$-modules.
\end{lem}
\begin{proof}
Let $P$ be a non-zero finitely generated graded projective
$A$-module. Consider the graded projective $A_0$-module $\overline{P}
= P\otimes_A A_0$ which is non-zero 
by the graded Nakayama lemma. We obtain a graded projective $A$-module $\overline{P}\otimes_{A_0}A$ fitting into a commutative diagram
\begin{displaymath}
\xymatrix{
& \overline{P}\otimes_{A_0}A \ar[d] \ar@{-->}[dl] &\\
P \ar[r] & \overline{P} \ar[r] & 0
}
\end{displaymath}
where the vertical morphism is the canonical one and the dashed arrow exists by projectivity of $\overline{P}\otimes_{A_0}A$. By construction the morphism $\overline{P}\otimes_{A_0}A \to P$ is surjective so it splits. But
\begin{displaymath}
\overline{P}\otimes_{A_0}A \otimes_A A_0 \cong \overline{P} = P\otimes_A A_0
\end{displaymath}
and so by another application of the graded Nakayama lemma we see
$P\cong \overline{P}\otimes_{A_0}A$ is induced up 
from a graded projective $A_0$-module proving the lemma.
\end{proof}

\begin{defn}
  For every graded projective $A$-module
  $Q$, we define summands $Q_{\prec i}$ and $Q_{\succcurlyeq i}$  
	with $Q_{\prec i}$ in $\Pl i$ and
  $Q_{\succcurlyeq i}$ in $\Pg i$ via the unique up to isomorphism split exact sequence of
  graded projective modules 
\[0 \to Q_{\prec i} \to Q \to
  Q_{\succcurlyeq i} \to 0\] which exists by the previous lemma.
\label{splitting_of_projectives}
\end{defn}

The next lemma is a key technical observation concerning the structure of resolutions 
over $A$.

\begin{lem}
\label{gregs_lemma}
 Every object $M$ in $\Db \grA$ is quasi-isomorphic to a complex of
 finitely generated graded projective $A$-modules
\begin{displaymath}
P = \cdots \to P^j \to  P^{j+1} \to \cdots
\end{displaymath}
such that $P^j = 0$ for all $j \gg 0$ and for any $i\in \Z$ there exists a $k_i$ with $P^{-k}$ 
 in $\Pg i$ for all $k \geq k_i$.
\end{lem}
\begin{proof}
It is sufficient to prove the result for graded $A$-modules as the condition 
is closed under suspensions and taking cones, and every object of 
$\Db \grA$ can be written as an iterated extension of suspensions of modules using the standard 
t-structure. Let us introduce notation local to
this proof. Given 
a finitely generated $A$-module $M$ define the integer
\begin{displaymath}
\min(M) = \min\{i\in \Z\; \vert \; M_i\neq 0\}.
\end{displaymath}

Let $M$ be a finitely generated graded $A$-module and set
\begin{displaymath}
\overline{M} = M\otimes_AA_0 = M/A_{\geq 1}M,
\end{displaymath}
which we consider as a graded $A_0$-module. We may assume $M$ has infinite projective dimension 
as the result is trivial in the finite projective dimension case. 
We will construct a projective resolution of the desired form. 
If $\overline{M}$ is zero then, by Nakayama, so is $M$ and thus we may 
suppose $\overline{M} \neq 0$. We choose an epimorphism from a graded projective 
$A_0$-module $\overline{P^0} \to \overline{M}$ by writing 
$\overline{M} \cong \oplus_{i=1}^n M_i(a_i)$,
taking epimorphisms $\overline{P^0_i}(a_i)\to M_i(a_i)$ where the $\overline{P^0_i}$ are projective 
$A_0$-modules and setting $\overline{P^0} = \oplus_{i=1}^n \overline{P^0_i}(a_i)$ with 
the obvious morphism to $\overline{M}$. This gives rise to an exact sequence of graded $A$-modules
\begin{displaymath}
0 \to Z^0 \to P^0 \to M \to 0,
\end{displaymath}
where $P^0 = \overline{P^0}\otimes_{A_0}A$, with the property that
\begin{displaymath}
\min(Z^0) \geq \min(P^0) = \min(M).
\end{displaymath}
We have assumed $A_0$ has finite global dimension, say $d$. Proceeding as above 
we may find projectives $P^i$ for $i=1,\ldots,d-1$ and exact sequences
\begin{displaymath}
0 \to Z^{i} \to P^{i} \to Z^{i-1} \to 0
\end{displaymath}
with $\min(Z^i)\geq \min(P^i) = \min(Z^{i-1})$.
Thus, restriction to the graded components in degree $j=\min(M)$ gives an exact sequence
\begin{displaymath}
0 \to Z^{d-1}_j \to P^{d-1}_j \to \cdots \to P^0_j \to M_j \to 0
\end{displaymath}
of $A_0$-modules with the $P^i_j$ projective. As $A_0$ has global dimension $d$ we see $Z^{d-1}_j$ must be projective. 
Hence $\overline{Z^{d-1}}$ can be written as $Z^{d-1}_j\oplus X$ with $X$ living in degrees strictly greater than $j$. As 
before we can pick an epimorphism $Q\to X$ from a graded projective $A_0$-module $Q$ which lives in the same 
degrees as $X$. Setting $P^{d} = (Z^{d-1}_j\oplus Q)\otimes_{A_0}A$ we get a short exact sequence
\begin{displaymath}
0 \to Z^{d} \to P^{d} \to Z^{d-1} \to 0
\end{displaymath}
with $\min(Z^d) > \min (M)$; thus our recipe guarantees projectives with 
generators in degrees less than or equal to $\min(M)$ cannot occur beyond the $d$th step 
of the resolution. We can now repeat this procedure starting at $Z^d$ to obtain a resolution satisfying 
the desired properties.
\end{proof}

\begin{rem}
It is easy to construct examples showing this lemma is no longer true
if $A_0$ does not have finite global dimension. Indeed, let $A =
k[x,y]/(x^2, y^2)$, with $|x| = 0$ and $|y| = 1$. The resolution of
$A/(x)$ is
\[ \ldots \to A \xra{x} A \xra{x}  A \to 0\]
which does not satisfy the conclusion of the previous lemma.
\end{rem}

\begin{proof}[Proof of \ref{Pl_admiss}]
Given an object $M$ in $\Db \grA$, let
$P \xra{\simeq} M$ be a quasi-isomorphism where $P$ is a complex of projectives as in the previous
lemma. Apply the decomposition \ref{splitting_of_projectives} degree-wise to $P$ to get a triangle
\[ P_{\prec i} \to P \to P_{\succcurlyeq i} \to \]
where $P_{\prec i}$ is the subcomplex of $P$ consisting of all projective
  summands generated in degrees less than $i$ and $ P_{\succcurlyeq
    i}$ is the quotient complex consisting of all projective summands
  generated in degree at least $i$.
Since $P^{-k}$ is
 in $\Pg i$ for all $k \gg 0$, we see that $P_{\prec i}$ is bounded, and
 hence in $\Pl i$. Note that $P_{\succcurlyeq i}$ has bounded finitely
 generated cohomology by the triangle, and so must be in $\Db \grAi$.

Now observe that there are no non-zero maps from objects in $\Pl i$ to
any module $M$ in $\grAi$. Thus $\grAi$ is contained in ${\Pl i}^\perp$
and hence so is $\Db \grAi$ since it is generated by $\grAi$ and ${\Pl i}^\perp$ is thick. Thus $\Pl i
\subseteq {}^\perp{\Db \grAi}$. We can now apply Lemma \ref{crit-for-so-decomp}.
\end{proof}

\begin{rem} \label{counit_pl_admiss} Let $M$ be an object of $\Db
  \grA$ and $P$ a projective resolution of
  $M$ satisfying the conditions of Lemma~\ref{gregs_lemma}. The proof
  shows that the localization triangle for $M$ is given by
\[ P_{\prec i} \to P \simeq M \to P_{\succcurlyeq i} \to .\]
\end{rem}

\begin{rem}
Although we have chosen to work throughout with the grading group $\Z$, the 
results are valid more generally. One can replace $\Z$ by any totally ordered 
abelian group and work with graded rings concentrated in degrees greater than or equal to the identity.

This will also be the case for the majority of the results that follow. However, there are instances in which one does need additional hypotheses. 
For example in Lemma~\ref{except-objects-so-decomp-lemma} (and the main theorem \ref{main-thm}) one must assume the order admits 
successors.
\end{rem}

\section{Non-commutative Proj and local cohomology}
For a graded non-commutative ring $A$, Artin and Zhang in
\cite{MR1304753} defined the
category of quasi-coherent sheaves on the non-commutative projective scheme
$\Proj A$ as the category of graded modules modulo the full
subcategory of torsion modules (here and throughout torsion means torsion with respect to the two-sided ideal $A_{\geq 1}$). In this section we recall some of their definitions and
results, in particular concerning local cohomology functors. When
$A$ is Gorenstein, these give a semiorthogonal
decomposition of $\Db \grAi$ that is a key step in the proof of Orlov's
theorem. 

We assume $A
= \bigoplus_{i \geq 0} A_i$ is a positively graded
right noetherian ring. We consider $\GrA$,
the abelian category of graded right $A$-modules. This contains $\grA$, the category of
finitely generated graded $A$-modules, as a full abelian subcategory. 

\begin{defn}
  Let $M$ be a graded $A$-module. An element $m \in M$ is \emph{torsion} if 
$$m\cdot(A_{\geq n}) = 0$$ for some $n \geq
  1$. Denote by $\tors M$ the submodule of $M$ consisting of all
  torsion elements. The module $M$ is \emph{torsion} if $\tors M = M$ and
  \emph{torsion-free} if $\tors M = 0$. Denote by $\TorsA$ the full
  subcategory of $\GrA$ consisting of torsion modules and set $\torsA = \TorsA \cap
  \grA$.
\end{defn}

The subcategory $\TorsA$ (respectively, $\torsA$) satisfies the property
that for a short exact sequence
\[ 0 \to X' \to X \to X'' \to 0\]
in $\GrA$ ($\grA$), we have $X$ in $\TorsA$ ($\torsA$) if and only
if $X'$ and $X''$ are in $\TorsA$ ($\torsA$) i.e., they are Serre 
subcategories. Moreover, $\TorsA$ is closed under colimits. Thus we can form the
quotient categories
\[ \qcoh X = \GrA/\TorsA \qquad \text{and} \qquad \coh \X = \grA /
\torsA,\]
see e.g.\ \cite[\S 4.3]{MR0340375} for the construction. The relevant
features here are that:
\begin{enumerate}
\item the categories $\qcoh X$ and $\coh X$ have the
  same objects as $\GrA$ and $\grA$, respectively;
\item the categories $\qcoh X$ and $\coh X$ are abelian and there are canonical exact functors $\GrA \to \qcoh X$ and $\grA
  \to \coh X$;
\item a map $f$ in $\GrA$ is an isomorphism in $\qcoh X$ if and only
  if $\ker f$ and $\operatorname{coker} f$ are in $\TorsA$; in particular the image
  of every object in $\TorsA$ is isomorphic to zero in $\qcoh X$. The
  analogous statement holds for $\grA$.
\end{enumerate}
For an object $M$
in $\GrA$, we denote by $\twid M$ the image of $M$ in $\qcoh X$. For
future reference, we note that as $\TorsA$ is closed under the grading
shifts, the shifts induce automorphisms of $\qcoh X$ and $\coh X$ which we also denote by $(-)(i)$.

\begin{rem}
  The notation $\qcoh X$ and $\coh X$ reflects that these categories should be thought of as
  sheaves of modules on the noncommutative projective scheme $X=\Proj
  A$. If $A$ is commutative and generated in degree 1, then by a famous result of Serre, 
  the category $\qcoh X$ (respectively $\coh X$) is equivalent to the
  category of quasi-coherent (respectively coherent) sheaves on the scheme
  $X = \Proj A$. If $A$ is generated in higher degrees, then $\coh X$
  is equivalent to the category of coherent sheaves on the
  Deligne-Mumford stack $\Proj A$.
\end{rem}

\begin{defn}
  For $M,N$ in $\GrA$, denote by $\uHom \GrA M N$ the graded abelian group
  \[\uHom {\GrA} M N = \bigoplus_{i \in \Z} \Hom {\GrA} {M(-i)} N
  \cong \bigoplus_{i \in \Z} \Hom {\GrA} {M} {N(i)}.\]
If $M$ is an $A$-$A$-bimodule, e.g.\ $M = A$, then $\uHom
{\GrA} M N$ is a graded right $A$-module and so is in $\GrA$.
\end{defn}

For any integer $p \geq 0$, we have a short exact sequence of
$A$-bimodules:
\[ 0 \to A_{\geq p} \to A \to A/A_{\geq p} \to 0.\]
Applying $\uHom \GrA - - $, we have an exact sequence of
endofunctors on $\GrA$:
\[ 0 \to \uHom \GrA {A/A_{\geq p}} - \to \uHom \GrA A - \to \uHom \GrA
{A_{\geq p}} - .\]
We may take the colimit of these sequences as $p \to \infty$ to get
another exact sequence of functors; the sequence remains exact as both the abelian 
structure and colimits for endofunctors are inherited value-wise from $\GrA$ and $\GrA$ 
has exact filtered colimits. Note
that for any $M$ in $\GrA$ we have isomorphisms in $\GrA$:
\[\displaystyle \colim_{p \to \infty} \uHom \GrA
{A/A_{\geq p}} M \cong \tors M \text{ and } \uHom \GrA A M \cong M.\] 
This gives
a functorial exact sequence
\begin{equation}
0 \to \tors M \to M \to \colim_{p \to \infty}\uHom \GrA
{A_{\geq p}} M.\label{counit-unit-for-qcoh}
\end{equation}

\begin{prop}
The inclusion of $\TorsA$ into $\GrA$ and the corresponding quotient
functor have right adjoints $\tors -$ and $\Gamma_*$, respectively:
\[\xymatrix{ \TorsA \ar@<1ex>[r]^-{\operatorname{inc}} & \GrA \ar@<1ex>[r]^-{\twid{(-)}}
  \ar@<1ex>[l]^-{\tors -} & \qcoh X
  \ar@<1ex>[l]^-{\Gamma_*}}\]
where for $M$ in $\GrA$,
$\tors M$ is the torsion submodule of $M$ and 
\[\Gamma_*(\twid M) = \displaystyle\colim_{p \to \infty}\uHom \GrA
{A_{\geq p}} M.\]

The functor $\Gamma_*$ and the inclusion of $\TorsA$ are fully faithful so the corresponding counit and unit respectively are isomorphisms. The remaining
counit and unit are given by \eqref{counit-unit-for-qcoh}.
\end{prop}

\begin{proof}
  It is easy to see that $\tors {-}$ is a right adjoint to the inclusion of
  $\TorsA \to \GrA$ and it follows from abstract nonsense, see \cite[\S
  4.4]{MR0340375}, that there exists a right adjoint $\Gamma_*\colon \qcoh
  X \to \GrA$. We give here a direct proof that the functor
  $\displaystyle \colim_{p \to \infty}\uHom \GrA
{A_{\geq p}} -$ induces a right adjoint, and that the unit is given by \eqref{counit-unit-for-qcoh}.

We first prove two claims:
\begin{claim}{1}
 If $M$ is in $\TorsA$, then $\displaystyle \colim_{p \to \infty}\uHom \GrA
{A_{\geq p}} M = 0$.
\end{claim}
\begin{claimproof}
  To see this, let $\phi$ be an element of $\uHom \GrA {A_{\geq p}}
  M$, for some $p \geq 0$. As $A$ is right noetherian, $A_{\geq p}$ is
  finitely generated as a right ideal by some $x_1, \ldots, x_k$. We can find $m \geq 0$ so
  that $\phi(x_i) \cdot A_{\geq m} = 0$ for all $i$, using that $M$ is torsion. Since $\phi(x_i \cdot A_{\geq m})
  = \phi(x_i) \cdot A_{\geq m} =0$ and, picking $m$ larger if necessary, $A_{\geq m +
    p} = (x_1, \ldots, x_k) A_{\geq m}$, we have that $\phi |_{A_{\geq m +
      p}} = 0$ and so $\phi = 0$ in $\displaystyle \colim_{p \to
    \infty}\uHom \GrA {A_{\geq p}} M$.
\end{claimproof}
\begin{claim}{2}
 There are no non-zero morphisms from torsion
modules to modules in the image of $\displaystyle \colim_{p \to \infty}\uHom \GrA
{A_{\geq p}} -$.
\end{claim}

\begin{claimproof}
 Let $T$ be in $\TorsA$ and let $g: T \to \displaystyle \colim_{p \to
  \infty} \uHom \GrA {A_{\geq p}} N$ be a map. For $x \in T$, let
$\phi \in \Hom \GrA {A_{\geq p}} N$ be a representative of $g(x)$, for
some $p$. Pick
$m$ such that $x \cdot A_{\geq m} = 0$. We have that $g(x) \cdot A_{\geq m} = g( x
\cdot A_{\geq m} ) = 0$, and that $\phi \cdot A_{\geq m}$ represents
$g(x) \cdot A_{\geq m} $. However, picking a larger $m$ if necessary, we see $\phi \cdot A_{\geq m}$ is the
image of $\phi$ under the map $\uHom \GrA {A_{\geq p}} N \to \uHom
\GrA {A_{\geq m+p}} N$ and so $\phi = 0$ in $\displaystyle \colim_{p \to
  \infty} \uHom \GrA {A/A_{\geq p}} N$, i.e.\ $g(x) = 0$. 
\end{claimproof}

To see that $\displaystyle \colim_{p \to
  \infty} \uHom \GrA {A_{\geq p}} -$ induces a functor $\Gamma_*:
\qcoh X \to \GrA$
 it is enough to show that $\displaystyle \colim_{p \to
  \infty} \uHom \GrA {A/A_{\geq p}} -$ takes morphisms $f$ with $\ker
f$ and $\coker f$ in $\TorsA$ to invertible morphisms. This follows
from Claims 1 and 2, and two applications of the snake lemma.
 To show that $\Gamma_*$ is right adjoint to the quotient and
\eqref{counit-unit-for-qcoh} is the unit, it is enough to show that
any map $f: M \to \Gamma_*( \twid N)$ factors through $M \to \Gamma_*(
\twid M)$. Note that by construction, we may extend
\eqref{counit-unit-for-qcoh} to an exact sequence
\begin{multline}
  \label{eq:ext_seq_qcoh}
    0 \to \displaystyle \colim_{p \to \infty}\uHom \GrA {A/A_{\geq p}}
    M \to M \to \\ \displaystyle \colim_{p \to \infty} \uHom \GrA {A_{\geq
        p}} M \to \displaystyle \colim_{p \to \infty} \uExt 1 \GrA {A/A_{\geq p}} M \to 0.
\end{multline}
Since $\uHom \GrA {A/A_{\geq p}} -  \cdot A_{\geq p} = 0$,
subobjects and quotients of torsion modules are torsion, and the
colimit of torsion modules is torsion, we see that the last term $\displaystyle \colim_{p \to
  \infty} \uExt 1 \GrA {A/A_{\geq p}} M$ is in $\TorsA$. To see that
the map $f: M \to \Gamma_*( \twid N) = \displaystyle \colim_{p \to
  \infty} \uHom \GrA {A_{\geq p}} N$ factors through $M \to \displaystyle \colim_{p \to
  \infty} \uHom \GrA {A_{\geq p}} M$, by \cite[4.1]{MR0340375}, it
is enough to show that there are no non-zero morphisms from torsion
modules to modules in the image of $\Gamma_*$, which was shown in
Claim 2.

We now show $\Gamma_*$ is fully faithful. Let $\eta_M: M \to \Gamma_*
\twid M$ be the unit of the adjunction, which is the center
arrow of \eqref{eq:ext_seq_qcoh}. Since the outer two terms of that
sequence are torsion, it follows that $\twid{\eta_M} $ is an isomorphism. Let $\epsilon_{\twid M}:
\twid {\left(\Gamma_* \twid M \right)} \to \twid M$ be the counit of
the adjunction. By
definition, the composition 
\[\twid M
\xra{\twid{\eta_M}} \twid {\left(\Gamma_* \twid M \right)}\xra{\epsilon_{\twid M}}
\twid M\] is the identity. Thus $\epsilon_{\twid M}$ is an isomorphism
and so $\Gamma_*$ is fully faithful.
\end{proof}

\begin{rem}
 As the notation suggests, if $A$ is a commutative ring generated in
 degree 1, then $\Gamma_*( \twid -)$ is isomorphic to $\bigoplus_{i
   \in \Z} \Gamma( \Proj A, \twid {(-)} (i))$ as they are both right adjoint to
 sheafification $\GrA \to \qcoh X$.
\end{rem}

It is clear from the definition that the functor $\tors -$ takes $\grA$ to
$\torsA$. However, $\Gamma_*$ does not necessarily take objects of $\coh X$
to $\grA$:
\begin{ex}
 Let $A = k[x]$ with $k$ a field, graded by $|x| = 1$. The $A$-module structure on $\displaystyle \colim_{p \to
  \infty} \uExt 1 \GrA {A/A_{\geq p}} A$ is easily computed: it has a $k$-basis $e_1,
\ldots, e_n, \ldots$ with $|e_i| = -i$ and $x e_i = e_{i - 1}$. In
particular it is not finitely generated over $A$ and so
from \eqref{eq:ext_seq_qcoh} we see that
$\Gamma_*(\twid A)$ is not either.
\end{ex}

In the example above, $(\Gamma_*(\twid A))_{\geq i}$ is finitely
generated (in fact of finite length) for any $i \in \Z$. Artin and Zhang gave a criterion
for $A$-modules that is equivalent to this fact being true. It is
often easy to check. For instance, it holds for all modules over
commutative rings.
\begin{defn}[Artin, Zhang]
 An object $M$ in $\grA$ satisfies $\chi_j(M)$ if there exists an integer $n_0$ such that $\uExt
 k \GrA {A/A_{\geq n}} M_{\geq i}$ is a finitely generated
 $A$-module for all $i\in \Z$, $k \leq j$ and all $n \geq n_0$. The ring satisfies condition $\chi_j$ if $\chi_j(M)$
 holds for all $M \in \grA$.
\end{defn}
If $M$ satisfies $\chi_1(M)$, then \cite[3.8.3]{MR1304753} shows that
$\displaystyle \colim_{p \to \infty} \uExt 1 \GrA {A/A_{\geq p}}
M_{\geq i}$ is a finitely generated $A$-module for all $i \in
\Z$. Thus by \eqref{eq:ext_seq_qcoh}, we see that $\Gi( \twid M ) := (\Gamma_*(\twid
M))_{\geq i}$ is finitely generated.

\begin{rem}
 As \cite[3.1.4]{MR1304753} shows, if $A$ is commutative, then every module $M$ satisfies
 $\chi_j(M)$. Indeed, we can compute the $A$-module $\Ext j \grA
 {A/A_{\geq p}} M$ using a graded free resolution of $A/A_{\geq p}$,
   which we can assume to be finite in each degree. If $A$ is not
   commutative then we must use the bimodule structure on $A/A_{\geq
     p}$ to compute the $A$-module structure on $\Ext j \grA
   {A/A_{\geq p}} M$, i.e.\ in this case we must look at the derived
   functor of $\Hom \grA {A/A_{\geq p}} -$ (rather than deriving in the first variable) and so we cannot
   necessarily use a free resolution of ${A/A_{\geq p}}$ to compute the $A$-module structure
   of $\Ext j \grA {A/A_{\geq p}} M$. In \cite{MR1291750}, an example
   is given of a non-commutative
   graded noetherian domain $A$ such that $\chi_j(A)$ does not hold
   for any $j > 0$.
\end{rem}

Recall that $\grAi$ is the full subcategory of $\grA$ with objects
those $M$ with $M = M_{\geq i}$. We denote by $\GrAi$ the analogous
subcategory of $\GrA$. Let $\TorsAi = \GrAi \cap \TorsA$ and $\torsAi
= \grAi \cap \torsA$. The functor $\tors -$ restricted to $\GrAi$
(respectively, $\grAi$) is a right adjoint of the inclusion $\TorsAi
\to \GrAi$ (respectively, $\torsAi \to \grAi$). Moreover, it is easy
to check that the composition of the functors $\GrAi \to \GrA \to
\qcoh X$ induces an equivalence $\GrAi / \TorsAi \xra{\cong} \qcoh
X$ and $\Gi = (\Gamma_*(-))_{\geq i}$ is a right adjoint to the quotient map. There is also an equivalence $\grAi/ \torsAi
\xra{\cong} \coh X$.

Assume that $A$ satisfies the condition $\chi_1$. Then, using the above,
we have the following diagram where
the vertical arrows are inclusions and the horizontal arrows form
adjoint pairs with the left adjoint on top:
\begin{equation}
\xymatrix{ \TorsAi \ar@<1ex>[rr]^{\operatorname{inc}} && \GrAi \ar@<1ex>[rr]^{\twid{(-)}}
  \ar@<1ex>[ll]^{\tors -} && \qcoh X
  \ar@<1ex>[ll]^{\Gi} \\
\torsAi \ar[u] \ar@<1ex>[rr]^{\operatorname{inc}} && \grAi \ar@<1ex>[rr]^{\twid{(-)}}
\ar[u] \ar@<1ex>[ll]^{\tors -}&&
  \ar@<1ex>[ll]^{\Gi}\coh X \ar[u].}\label{eq:abel-localiz}
\end{equation}
For any $M$ in $\GrAi$, one counit and one unit are isomorphisms and the other two are given
by
\begin{equation}
0 \to \tors M \to M \cong \uHom \GrA A M \to \displaystyle (\colim_{p
  \to \infty} \uHom \GrA {A_{\geq p}} M)_{\geq i}\label{eq:GrAi-counit-unit}
\end{equation}
as in the case of $\GrA$.
Also note that $\Gi$ is fully faithful, as it is the right adjoint of a quotient functor.

We wish to extend this diagram to functors between the bounded derived
categories of the above abelian categories. The existence of a corresponding 
localization sequence involving the derived categories is standard but we provide 
some details. We start with a simple lemma.

\begin{lem}
 Let $M$ be an object of $\GrAi$ and let $M \to I$ be an injective
 resolution in $\GrA$. Then $M \to I_{\geq i}$ is an injective
 resolution in $\GrAi$. In particular $\GrAi$ has enough injectives.
\end{lem}

\begin{proof}
 The functor $(-)_{\geq i}$ is exact and $M_{\geq i} = M$, thus $M \to
 I_{\geq i}$ is a quasi-isomorphism. So to complete the proof it is sufficient to 
show $I_{\geq i}$ is a complex of injectives.

 Let $J$ be an injective object
 in $\GrA$. By adjunction there is an isomorphism of functors $\Hom {\GrA} {\operatorname{inc}(-)} {J} \cong
 \Hom {\GrAi} {-} {J_{\geq i}}$. The former functor is exact as $J$ is injective and the inclusion is exact, and thus so is the latter showing $J_{\geq i}$ is
 injective in $\GrAi$. 

Thus $(-)_{\geq i}$ preserves injectives and so the quasi-isomorphism $M\to I_{\geq i}$ is an injective resolution.
\end{proof}

The functors to the right in \eqref{eq:abel-localiz} are
exact and those to the left are left exact (since they are right
adjoints). Since $\GrAi$ has enough injectives by the above lemma and $\qcoh X$ has enough
injectives by \cite[7.1]{MR1304753} (in fact, by standard abstract nonsense both of these categories are Grothendieck categories and so have enough injectives), we may form $\R \tors -$ and $\R \Gi$, the right
derived functors of $\tors -$ and $\Gi$, respectively. This gives two
pairs of adjoint functors
\begin{equation}
\label{adjunction-big-derived-cats}
\xymatrix{ \mathsf{D}_{\TorsAi}(\GrAi) \ar@<1ex>[rr]^-{\operatorname{inc}} && \Dp \GrAi
  \ar@<1ex>[rr]^{\twid{(-)}} \ar@<1ex>[ll]^{\R \tors -} &&
  \ar@<1ex>[ll]^{\R\Gi}\Dp {\qcoh X}}
\end{equation}
where $\mathsf{D}_{\TorsAi}(\GrAi)$ is the full subcategory
of $\Dp \GrAi$ consisting of complexes with torsion cohomology.

Since $\Gi$ sends injectives to injectives and is fully faithful one checks easily that $\R\Gi$ is 
also fully faithful. In particular, we have that $\twid{(-)}$ is a quotient functor. As $\twid{(-)}$ at the level of the abelian categories is exact, 
the kernel of this functor at the level of derived categories consists of precisely those complexes whose cohomology is annihilated by $\twid{(-)}$ 
i.e., it is exactly $\mathsf{D}_{\TorsAi}(\GrAi)$. This proves the above functors give a localization sequence 
of triangulated categories.

It follows that for every $M\in \Dp \GrAi$ there is a localization triangle
\begin{equation}
\R \tors M \to M \to \R \Gi( \twid M) \to .\label{lc_tri}
\end{equation}
where the first map is the counit of the first adjunction of
\eqref{adjunction-big-derived-cats} and the second map is the unit of
the second adjunction of \eqref{adjunction-big-derived-cats}.

\begin{rems}
 Note that when $A$ is commutative, a triangle such as \ref{lc_tri} can be explicitly constructed using the
 Cech complex.
\end{rems}

For the above adjoint
pairs to restrict to the bounded derived categories of complexes
of finitely generated modules, we need to place two further
restrictions on $A$. Let $\R \Gamma_*\colon \Dp {\qcoh X} \to \Dp \GrA$ be the right derived
functor of the left exact functor $\Gamma_*$.
\newcommand{\cd}[1]{\operatorname{cd}(#1)}
\begin{defn}[Artin-Zhang]
 The \emph{cohomological dimension} of $A$ is
\[ \cd A  := \sup \{ d \, | \, H^d \R \Gamma_*( \twid A
) \neq 0 \}.\]
\end{defn} 
\noindent By \cite[7.10]{MR1304753}, if $\cd A$ is finite, then $\R \Gamma_*(\twid
M)$ is a bounded complex for every $\twid M \in \qcoh X$ and so
restricts to a functor
\[ \R \Gamma_*: \Db {\qcoh X} \to \Db \GrA.\]
Since $\Gi$ is the composition of $\Gamma_*$ and the exact
functor $(-)_{\geq i}$, we see that $\R \Gi = (\R \Gamma_*)_{\geq i}$
and so $\R \Gi$ restricts to a functor
\[\R \Gi: \Db {\qcoh X} \to \Db \GrAi.\]
By the long exact sequence in homology induced by \eqref{lc_tri}, we
see that $\R \tau$ also restricts to a functor between bounded derived
categories.

Now we consider finiteness. We want to compute the cohomology of $\R
\tors M$. We view $\tors - = \displaystyle \colim_{p \to \infty} \uHom
{\GrA} {A/A_{\geq p}} -$ as a functor from $\grAi \to \torsAi$. For
$M$ in $\grAi$, let $M \to I_M$ be an injective resolution in
$\GrA$. Then $(I_M)_{\geq i}$ is an injective resolution in $\GrAi$. Thus
we have \[\R \tors M = \colim_{p \to \infty} \uHom
{\GrA} {A/A_{\geq p}} {(I_M)_{\geq i}} = \colim_{p \to \infty} \uHom
{\GrA} {A/A_{\geq p}} {I_M}_{\geq i}\]
where the second equality follows from the commutativity of the square of inclusions
\[ \xymatrix{ \TorsAi \ar[r] \ar[d] & \GrAi \ar[d] \\
\TorsA \ar[r] & \GrA
}\]
by taking right adjoints. This shows that
 \[ H^k \R \tors M \cong \colim_{p \to \infty} \uExt k {\GrA} {A/A_{\geq p}} {M}_{\geq i}\]
for all $k \geq 0$. By \cite[3.8.3]{MR1304753}, if $M$ satisfies
$\chi_j(M)$, then $\colim_{p \to \infty} \uExt k {\grA} {A/A_{\geq p}}
{M}_{\geq i}$ is a finitely generated $A$-module for all $k \leq j$
and all $p \in \Z$. 

Assume now that $A$ has finite cohomological dimension and satisfies $\chi_j$ for all $j \geq 0$. The above
shows that $\R \tors -$ restricts to a functor
\[ \R \tors - : \Db \grAi \to \mathsf{D}^\mathsf{b}_{\torsAi}(\grAi).\]
By the long exact sequence in cohomology coming from \eqref{lc_tri},
we see that we also have a functor
\[\R\Gi: \Db {\coh X} \to \Db \grAi.\]
This gives the following diagram of adjoint functors, where the top functors are the left adjoints:
\[ \xymatrix{ \mathsf{D}^\mathsf{b}_{\torsAi}(\grAi) \ar@<1ex>[rr]^-{\operatorname{inc}} && \ar@<1ex>[ll]^-{\R
\tors -} \ar@<1ex>[rr]^-{\twid {(-)}} \Db \grAi 
 && \Db {\coh X} \ar@<1ex>[ll]^-{\R
\Gamma_{\geq i}(-)}}.\]
The functor $\R
\Gamma_{\geq i}$ is fully faithful and its image is left
admissible. Also, any object in this image is contained in $\left (
  \mathsf{D}^\mathsf{b}_{\torsAi}(\grAi) \right )^\perp$. Indeed, for
$M$ an object with torsion cohomology and any $N \in \Db \grAi$, we
have that
\[ \Hom {\Db \grAi} M {\R\Gi {\twid N}} \cong \Hom {\Db {\coh X}}
  {\twid M} {\twid N} = 0\]
since $\twid M \simeq 0$. From this containment and the triangle
\eqref{lc_tri}, we may apply \ref{crit-for-so-decomp} to see that
there is a semiorthogonal decomposition
\[ \Db \grAi = \left (\R \Gamma_{\geq i} \Db
  {\coh X}, \mathsf{D}^\mathsf{b}_{\torsAi}(\grAi) \right ).\]

Recall that $\Sg i$ is the thick subcategory generated by $A_0(e)$ for
all $e \leq -i$.
\begin{lem}
 There is an equality $\Sg i = \mathsf{D}^\mathsf{b}_{\torsAi}(\grAi).$ \label{sgi-tors-lemma}
\end{lem}
\begin{proof}
 It's clear that $A_0(e)$ is in $\torsAi $ for all $e \leq
 -i$, so $\Sg i$ is contained in
 $\mathsf{D}^\mathsf{b}_{\torsAi}(\grAi)$. Given $M$ in
 $\mathsf{D}^\mathsf{b}_{\torsAi}(\grAi)$, we have that $H^*(M)$ is
 finitely generated and torsion, thus $M$ must have cohomology in only
 finitely many degrees. Analogously to the proof of \ref{desc_Sl}, this
 shows that $M$ is in $\Sg i$.
\end{proof}

The above shows the following:
\begin{prop}
\label{sgi_admiss}
  Let $A$ be a positively graded right
noetherian ring that satisfies condition $\chi$ and has finite
cohomological dimension. Then there is a semiorthogonal decomposition 
\[ \Db \grAi = \left (\R \Gamma_{\geq i}\Db
  {\coh X} , \Sg i \right ).\] The corresponding localization triangle is given by \eqref{lc_tri}.
\end{prop}

\section{Singularity category of a Gorenstein ring}
In this section we assume that $A = \bigoplus_{i \geq 0} A_i$ is a positively graded (two-sided)
noetherian ring with $A_0$ of finite global dimension, but not
necessarily commutative.

In the following, we denote by $\id_A M$ the graded injective
dimension of a graded $A$-module $M$.
\begin{defn}
\label{defn_gor}
 The ring $A$ is \emph{(Artin-Schelter) Gorenstein} if $\id_A A \li$, $\id_{A^\op} A \li$
 and 
\[\RuHom \grA {A_0} A \cong A_0[n](a) \text{ for some }
n, a \in \Z\]
in both $\Db \grA$ and $\Db \grAop$. The unique integer $a$ is the \emph{$a$-invariant} of
 $A$.
\end{defn}

\begin{rem}
In \cite{MM} a different definition of Artin-Schelter Gorenstein ring is given under the restriction that $A_0$ is a finite dimensional algebra over a fixed base field $k$. Their definition differs from ours in two ways: Minamoto and Mori require the shift occurring to match the injective dimension of $A$ i.e., $n = -\id_A A$, and that rather than $\RuHom \grA {A_0} A \cong A_0[n](a)$ one asks for an isomorphism
\begin{displaymath}
\RuHom \grA {A_0} A \cong \Hom k {A_0} k[n](a).
\end{displaymath}
We note both definitions restrict to the classical one in the case $A_0 = k$.

As an example, if $R$ is a commutative regular ring of positive Krull dimension and we set $A = R[x]/(x^n)$ with $x$ in degree $1$ then $A$ is AS-Gorenstein in our sense but not according to \cite{MM}. On the other hand the definition of Minamoto-Mori covers certain (higher) preprojective algebras which are in general excluded by our definition.
\end{rem}

From this point forward we will use the term \emph{Gorenstein ring} to refer to a ring that is Gorenstein either in the sense of Definition~\ref{defn_gor} or \cite{MM}. Our results hold for both definitions. We will work with the definition we give and, when necessary, point out what changes in the arguments are necessary if one uses the definition of Minamoto and Mori. In fact, the only place in which the arguments do not go through verbatim are Lemmas \ref{lem_chi} and \ref{lem_more_decomp} which require minor tweaking.

The most important feature of Gorenstein rings for us is the duality
given below. We will make a standard abuse of notation and not differentiate between 
the two duality functors notationally.
\begin{lem}
\label{a_inv_lem} Assume that $A$ is a
Gorenstein ring. Then the functors
\[D = \RuHom\grA - A : \Db \grA \to
    {\Db \grAop}^{\op}\]
\[D = \RuHom\grAop - A :
    \Db \grAop \to {\Db \grA}^{\op}\] are quasi-inverse equivalences.
\end{lem}
\begin{proof}
	We first observe that $D$ does indeed take $\Db \grA$ to $\Db \grAop^\op$. Since $A$ has finite injective dimension as both a left and a right 
	module over itself it is clear that $D$ preserves boundedness of cohomology. It is also clear that $D$ sends complexes 
	with finitely generated cohomology groups to the same as we can resolve any object of $\Db \grA$ by a complex of finitely generated projectives and 
	$A$ is noetherian.
	
  These functors are adjoint so we can consider the unit of this adjunction
	\[\eta\colon \Id \to D^2\]
	and we need to show it is an equivalence. But this is again clear: for a bounded above complex of finitely generated projectives the 
	map $\eta$ is just componentwise the natural map to the double dual and finitely generated projectives are reflexive.
	
\end{proof}

Recall that $\Pg i$ is the thick subcategory of $\Db \grA$
generated by those $A(e)$ with $e \leq -i$ and $\Pl i$ is the thick
subcategory generated by the $A(e)$ with $e > -i$.

\begin{lem}
\label{prelim-so-decomp-grAi-pgi}
 If $A$ is Gorenstein, there is a semiorthogonal decomposition
\[ \Db \grAi = \left ( \Pg i,
  ({}^\perp\Pg i) \cap \Db \grAi \right ).\]
For $M \in \Db \grAi$, the localization triangle is given by
\[ D(G)_{\prec
  i} \to M \cong D(G) \to D(G)_{\succcurlyeq i} \to\]
where the notation is as in Definition~\ref{splitting_of_projectives}, and $G \to D(M)$ is a projective resolution of the dual of $M$ as in
Lemma~\ref{gregs_lemma}.
\end{lem}

\begin{proof}
Let $M$ be an object of $\Db \grAi$ and let $G \to D(M)$ be a projective resolution 
as in Lemma~\ref{gregs_lemma}, where $D(M) = \R \uHom \grA M A$ is the
image of $M$ under the duality functor. As in the proof of Lemma
\ref{Pl_admiss}, there is a triangle
\[ G_{\prec -i+1} \to G \to G_{\succcurlyeq -i+1} \to
\]
where $G_{\prec -i+1}$ is an object of $\Pl {-i+1}$ and every component of $G_{\succcurlyeq -i+1}$ is generated in degree at least $-i +
1$. If $P = P_0 \otimes_{A_0} A(e)$ is any indecomposable graded
projective $A$-module, then $D(P) \cong \uHom
\grA P A \cong (P_0)^* \otimes_{A_0} A(-e)$, where $(P_0)^*$ is the $A_0$-dual of $P_0$. Thus $D(G_{\prec-i+1}) =
D(G)_{\succcurlyeq i}$ and $D(G_{\succcurlyeq -i + 1}) = D(G)_{\prec
  i}$. Applying
$D$ to the triangle above gives a triangle
\[ D(G)_{\prec
  i} \to D(G) \to D(G)_{\succcurlyeq i} \to.\]
Note that $D(G)_{\succcurlyeq i} $ is in $\Pg i$ and that there are isomorphisms
$M \xra{\simeq}
D(D(M)) \xra{\simeq} D(G)$. We can now apply Lemma
\ref{crit-for-so-decomp}, once we show that $D(G)_{\prec
  i}$ is in $({}^\perp\Pg i) \cap \Db \grAi$. It follows from the long
exact sequence in homology of the above triangle that each of the
homology groups of $D(G)_{\prec
  i}$ is generated in degrees at least $i$ and thus $D(G)_{\prec
  i}$ is in $\Db \grAi$. That $D(G)_{\prec
  i}$ is in $^\perp(\Pg i)$ follows from the fact that $\Hom {\grA}
{A(e)} {A(f)} = 0$ for $e > f$.
\end{proof}

Let us denote by $\Bi$ the subcategory $({}^\perp\Pg i) \cap \Db \grAi$,
which by the above lemma is a right
admissible subcategory of $\Db \grAi$. There is a description of $\Bi$
using the well-known
singularity category of $A$.
\begin{defn}
\label{defn_sing_cat}
 Let $A$ be a graded ring.
 \begin{enumerate}
 \item An object $M$ in $\Db \grA$ is \emph{perfect} if $M$ is in the
   thick subcategory generated by $A(e)$ for all $e \in \Z$ i.e., it is quasi-isomorphic 
	to a bounded complex of projectives. We denote
   the subcategory of perfect complexes by $\perf A$. We see from the
   definitions that $\perf A = \gend{ \Pg i, \Pl i}$.
\item The \emph{singularity category} of $A$ is
\[ \Dsg \grA := \Db \grA / \perf A.\]
 \end{enumerate}
\end{defn}

Orlov showed that when $A$ is a connected graded Gorenstein algebra over a field, there is an embedding of
$\Dsg \grA$ in $\Db \grA$ for every $i \in \Z$, and the image is equal
to $\Bi$. We now show this holds in the generality in which we are working.
First we recall a lemma whose proof is left to the reader.

\begin{lem}
\label{lem-admiss-equivs}
Let $\mathcal{A}$ be a left admissible subcategory in a triangulated
category $\T$ with $i_L: \T \to \mathcal{A}$ the left adjoint to the
inclusion $i: \mathcal{A} \to \T$. Then $i_L$ induces an equivalence
\[ \T/{}^\perp\mathcal{A} \to \mathcal{A}\]
with inverse equivalence the composition $\mathcal{A} \to \T \to \T/{}^\perp\mathcal{A}.$
The analogous statement holds for right admissible subcategories.
\end{lem}
Applying the above lemma to \ref{Pl_admiss} shows that there is an equivalence
\[ \psi_i\colon \Db \grA /\Pl i \xra{\cong} \Db \grAi.\]
Remark \ref{counit_pl_admiss} shows that $\psi_i(M) = P_{\succcurlyeq i}$, where $P \to M$ is a projective resolution 
as in Lemma~\ref{gregs_lemma}.
If we apply Lemma \ref{lem-admiss-equivs} again to the semiorthogonal decomposition $\Db
\grAi = (\Pg i, \Bi)$, we have an equivalence
\[ \phi_i\colon \Db \grAi / \Pg i \xra{\cong} \Bi\]
with $\phi_i(N) = D(Q)_{\prec
  i}$, where $Q \to D(N)$ is a projective resolution 
as in Lemma~\ref{gregs_lemma}.
Let us set $\bb_i = \phi_i \circ \pi \circ \psi_i$ where $\pi\colon \Db \grAi\to \Db \grAi / \Pg i$ is the quotient functor. This gives an
equivalence
\begin{equation}
\bb_i\colon \Dsg \grA = \Db \grA/\gend{ \Pl i, \Pg i} \xra{\cong} \Bi\label{defn-bbi}
\end{equation}
with $\bb_i(M) = D(Q)_{\prec
  i}$, where $Q \to D(P_{\succcurlyeq i})$ and $P \to M$ are projective
resolutions as in Lemma~\ref{gregs_lemma}. The inverse of the
equivalence is given by the composition of the inclusion and quotient $\Bi \to \Db \grA \to \Db \grA/\perf A$. Moreover,
we have that $\bb_i$ followed by the inclusion $\Bi\to \Db \grAi$ is left adjoint to the quotient functor $\Db \grAi = \Db
\grA / \Pl i \to \Db \grA/\gend{\Pl i, \Pg i} = \Dsg \grA$.

To sum up, we have shown the following:
\begin{prop}
  \label{so-decomp-grAi-sing-cat}
 If $A$ is a graded Gorenstein ring, the
 quotient $\Db \grAi \to \Dsg \grA$ has a
 fully faithful left adjoint 
\[\bb_i: \Dsg \grA \to \Db \grAi.\] The image of $\bb_i$ is the
 subcategory $\Bi = ({}^\perp\Pg i) \cap \Db \grAi$ and there is a semiorthogonal decomposition:
\[ \Db \grAi = ( \Pg i, \Bi).\]
The localization triangle is described in \ref{prelim-so-decomp-grAi-pgi}.
\end{prop}

\section{Relating the bounded derived category of coherent sheaves and the
  singularity category}
In this section we prove the main theorem by comparing the
semiorthogonal decompositions constructed in the previous sections. We
assume that $A = \bigoplus_{i \geq 0} A_i$ is a positively graded
noetherian Gorenstein ring with $A_0$ a ring
of finite global dimension, but not necessarily commutative.

Gorenstein rings often satisfy the two properties we need to apply \ref{sgi_admiss}. 
\begin{lem}
 If $A$ is a Gorenstein ring, then $A$ has finite cohomological
 dimension.
\end{lem}

\begin{proof}
	We need to show 
	\[ \cd A  = \sup \{ d \, | \, H^d \R \Gamma_*( \twid A
	) \neq 0 \}<\infty.\]
	Since $A$ is Gorenstein 
we can choose a bounded injective resolution $I$ for $A$ as a right
$A$-module. Hence $\R\tors A = \tors I$ has bounded cohomology and the localization triangle
	\begin{displaymath}
		\R\tors A \to A \to \R\Gamma_*(\twid A) \to
	\end{displaymath}
	then implies $\R\Gamma_*(\twid A)$ also has bounded cohomology.
\end{proof}

We have remarked earlier that any commutative ring satisfies condition $\chi$. The next lemma gives some noncommutative and not necessarily graded connected examples.

\begin{lem}\label{lem_chi}
Let $k$ be a commutative ring and $A$ a flat Gorenstein $k$-algebra. Then $A$ satisfies condition $\chi$.
\end{lem}

\begin{proof}
As $A$ is flat over $k$ it follows that the enveloping algebra $A\otimes_k A^\op$ is flat over both $A$ and $A^\op$. Thus the restriction of scalars functors induced by the maps $A\to A\otimes_k A^\op$ and $A^\op \to A\otimes_k A^\op$ preserve injectives. Taking an injective resolution $I$ of $A$ over $A\otimes_k A^\op$ thus gives a bimodule resolution of $A$ which is an injective resolution as both a complex of left and of right $A$-modules.
	
	We may use such a resolution to compute $D = \RuHom\grA - A$ as $\uHom\grA - I$ and obtain the correct $A^\op$-module structure and similarly for the inverse duality functor; this is just a consequence of the fact that $I$ and $A$ are quasi-isomorphic as complexes of bimodules. Given a complex of injectives $M\in \Db \grA$ we now compute, using the duality of Lemma~\ref{a_inv_lem}, that there are quasi-isomorphisms of right $A_0$-modules
	\begin{align*}
		\uHom\grA {A_0} M &\cong \RuHom{\grA^\op} {\uHom\grA M I} {\uHom\grA {A_0} I} \\
		&\cong \uHom{\grA^\op} {P} {\uHom\grA {A_0} I} \\
		&\cong \uHom{\grA^\op} {P} {{}_\nu A_0[n](a)},
	\end{align*}
	where $P$ is a projective resolution of $\uHom\grA M I$ over $A^\op$ and $\nu$ is a twist by some, possibly non-trivial, automorphism which needs to be accounted for as we view ${\uHom\grA {A_0} I}$ as a bimodule rather than just a right module (see for example \cite[Lemma~2.9]{MM}). Now $\uHom{\grA^\op} {P} {\Sigma^n {}_\nu A_0(a)}$ is a complex of finitely generated $A_0$-modules and so in particular has finitely generated cohomology over $A_0$ and hence over $A$. In particular, if $M$ is an injective resolution of a right $A$ module $N$ this shows $\uExt i \grA {A_0} N$ is finitely generated over $A$ for all $i\in \Z$.
	
	It only remains to observe that $A/A_{\geq n}$ has a filtration, as bimodules, by copies of $A_0(j)$ for $j\in \Z$ and considering the corresponding long exact sequences shows $\RuHom\grA {A/A_{\geq n}} M$ has finitely generated cohomology for all $M \in \Db \grA$. Hence $A$ satisfies condition $\chi$.

\end{proof}

\begin{rem}
In the above lemma if $A$ is AS-Gorenstein in the sense of \cite{MM} then one has to replace ${}_\nu A_0(a)$ by $\Hom k {{}_\nu A_0(a)} k$ but this does not alter the argument as $\uHom{\grA^\op} {P} {\Sigma^n \Hom k {{}_\nu A_0(a)} k}$ is still a complex of finitely generated $A_0$-modules.
\end{rem}

\begin{thm}
\label{main-thm}
Let $A = \bigoplus_{i \geq 0} A_i$ be a positively graded
noetherian Gorenstein ring with $A_0$ of finite global dimension, but
not necessarily commutative. We assume in addition that $A$ satisfies condition $\chi$. Let $a$ be the $a$-invariant of $A$
defined in \ref{defn_gor}.

\begin{enumerate}
\item If $a > 0$, then for any $i \in \Z$ there is a semiorthogonal
  decomposition
\[ \Db {\coh X} = \left ( \cO(-i - a +1), \ldots, \cO(-i),
\twid{\Bi}\right ),\]
where $\cO(j)$ is the image of $A(j)$ in $\coh X$ and $\Bi$ is the
image of $\Dsg \grA$ under the fully faithful functor $\bb_i: \Dsg \grA
\to \Db \grAi$ described in \eqref{defn-bbi}.
\item If $a < 0$, then for any $i \in \Z$ there is a semiorthogonal
  decomposition
\[ \Dsg \grA = \left (  pA_0(-i), \ldots, pA_0(-i + a + 1), p
\R\Gamma_{\geq i-a} \Db
{\coh X} \right ),\]
where $p: \Db \grAi \to \Dsg \grA$ is the canonical quotient.
\item If $a = 0$, then for any $i \in \Z$ the
  functors $\twid{(-)} \bb_i: \Dsg \grA \to \Db {\coh X}$ and $p
\R\Gamma_{\geq i}: \Db {\coh
    X} \to \Dsg \grA$ are inverse equivalences.
\end{enumerate}
\end{thm}

Before beginning the proof, we need two lemmas. For the rest of the section we rely heavily on notation introduced
earlier: recall that $\Sl i$ (respectively $\Sg i$) is the thick subcategory generated by the
 objects $A_0(e)$, for all $e > -i$ (respectively $e \leq -i$) and
 $\Pl i$ (respectively $\Pg i$) is the thick subcategory generated by the
 objects $A(e)$ for all $e > -i$ (respectively $e \leq -i$).

\begin{lem}
\label{except-objects-so-decomp-lemma}
 Let $A$ be a graded ring.
 \begin{enumerate}
 \item For any $i \in \Z$, there is a semiorthogonal decomposition
\[ \Pg i = \left ( \Pg {i+1}, A(-i) \right ). \]
\item For any $i \in \Z$, there is a
  semiorthogonal decomposition
\[ \Sl {i+1} = \left ( \Sl i, A_0(-i) \right ).\]
 \end{enumerate}
\end{lem}

\begin{proof}
It is clear that $A(-i) \subseteq {}^{\perp}\Pg {i+1}$. We know that any object in $\Pg i$ is
isomorphic in $\Db \grAi$ to a bounded complex $X$ of finitely generated
graded projective modules and so we may restrict ourselves to working with 
such complexes. As in the proof of \ref{Pl_admiss}, using the structure of graded projectives
given in \ref{lem_projdesc} and the notation of Definition~\ref{splitting_of_projectives}, we see that there is a short exact sequence of
complexes, split in each degree
\[ 0 \to X_{\prec i+1} \to X \to  X_{\succcurlyeq i+1} \to 0\]
where $X_{\prec i+1}$ is the subcomplex of $X$ which is termwise the projective
summands generated in degree $i$ and $X_{\succcurlyeq
  i+1}$ is the quotient complex which is termwise all those projective summands generated
in degree at least $i+1$. This gives a triangle
\[ X_{\prec i+1} \to X \to  X_{\succcurlyeq i+1} \to \]
with $X_{\prec i+1}$ in the thick subcategory generated by $A(-i)$ and
$X_{\succcurlyeq i+1}$ in $\Pg {i+1}$. By Lemma \ref{crit-for-so-decomp} we have proved
part 1.

We have that $A_0(-i) \in {}^\perp\Sl i$, since $\R \Hom \grA {A_0(e)} {A_0(f)} \simeq
0$ for all $e < f$. Indeed, we may find a graded free resolution of
$A_0(e)$ that exists entirely in degrees at least $e$. For any $X \in
\Sl {i + 1}$, we have the triangle
\[ X_{\geq i} \to X \to X/X_{\geq i} \to\]
as in Lemma~\ref{sli_admiss}. Since $X_{\geq i+1}=0$ we see $X_{\geq i}$ has cohomology concentrated in grading degree $i$ and so is in the thick subcategory generated by $A_0(-i)$. On the other hand $X/X_{\geq i}$ is killed by $(-)_{\geq i}$ so is in $\Sl i$ by Lemma~\ref{desc_Sl}. Applying Lemma \ref{crit-for-so-decomp} now proves
part 2.
\end{proof}

For the sake of clarity we introduce the following notation for the next lemma. We denote by $\Sl i(A)$ and $\Sl i(A^\op)$ the thick subcategories generated by the $A_0(e)$, for all $e > -i$, in $\Db \grA$ and $\Db \grAop$ respectively. We use similar notation for $\Sg i$, $\Pg i$, and $\Pl i$ in order to indicate in which category we are working. 

\begin{lem}\label{lem_more_decomp}
 Under the hypothesis of Theorem~\ref{main-thm}, we have
${{\Sg i}^\perp =
 {}^\perp{\Pg {i+a}}}$
as subcategories of $\Db \grAi.$
\label{lem-sg-equals-pg}
\end{lem}

\begin{proof}
As $A$ is Gorenstein we have Grothendieck duality by Lemma~\ref{a_inv_lem}. We note that restricting the duality functor $D = \RuHom \grA - A $ to $\Sg i$
gives an equivalence
\[D\colon \Sg i(A) \xra{\cong} \left ( \Sl {-i - a +1}(A^\op) \right )^{\,\op}.\]
Indeed, one can see this simply by computing $D$ applied to the generators and observing equivalences send 
thick subcategories to thick subcategories.
Similarly we can also
restrict $D$ to get an equivalence
\[ D\colon \Pl {-i - a + 1}(A) \xra{\cong} \left(\Pg {i+a}(A^\op)\right)^{\, \op}.\]
By \ref{sli_admiss}, \ref{Pl_admiss}, and the definition of a
semi-orthogonal decomposition, we have in $\Db \grA$
\[^\perp\Sl {-i - a +1}(A) = \Db {\grA_{\geq -i-a+1}} = \Pl {-i - a +1}(A)^\perp.\] We thus have
\begin{align*}
  \Sg i(A)^\perp& \xra{\cong} \left ( \left( \Sl {-i - a +1}(A^\op) \right
    )^{\,\op}\right )^\perp = \left ( {}^\perp\Sl {-i - a +1}(A^\op) \right
  )^{\,\op}
= \left ( \Pl {-i - a +1}(A^\op)^\perp \right )^{\, \op} \\ &= ^\perp\!\!\left (
  \left ( \Pl {-i - a +1}(A^\op) \right )^{\, \op} \right ) \xra{\cong}
  {}^\perp\Pg {i+a}(A),
\end{align*}
i.e.\ the functor $D^2$, which is isomorphic to the identity functor, takes $\Sg i(A)^\perp$ to
$^\perp \Pg {i+a}(A)$ and hence these categories are equal.
\end{proof}

\begin{rem}
If $A$ is AS-Gorenstein in the sense of \cite{MM} then one needs a minor additional argument to prove the above lemma. We need to check $D(\Sg i(A))$, the thick subcategory of $\Db \grAop^\op$ generated by the $D(A_0(e))$ for $e\leq -i$, is $(\Sl {-i - a +1}(A^\op))^{\,\op}$. By definition
\begin{displaymath}
\RuHom {\grA} {A_0(e)} A \cong \Hom k {A_0} k[-n](-e+a)
\end{displaymath}
and it is sufficient to check this object generates the same thick subcategory as $A_0(-e+a)$ (of course we can ignore the degree shift). This follows essentially immediately from the equivalence
\begin{displaymath}
\Hom k - k \colon \Db {\mathrm{mod}\;A_0} \xra{\cong} \Db {\mathrm{mod}\;A_0^\op}^\op
\end{displaymath}
which sends the generator $A_0$ to $\Hom k {A_0} k$.
\end{rem}

\begin{proof}[Proof of Theorem \ref{main-thm}]
 Combining the decompositions of \ref{sli_admiss}, \ref{so-decomp-grAi-sing-cat} via \ref{build-up-of-so-decomps} there is a semiorthogonal
 decomposition
\begin{equation} 
\label{so-decomp-grA-using-dsg}
\Db \grA = \left (\Sl i, \Pg i, \Bi \right ).
\end{equation}
Similarly, by \ref{sli_admiss}, \ref{sgi_admiss} and \ref{build-up-of-so-decomps},
there is a semiorthogonal decomposition
\begin{equation*}
 \Db \grA = \left ( \Sl i, \R\Gamma_{\geq i} \Db {\coh X}, \Sg i
 \right ).
\end{equation*}
Using Lemma \ref{lem-sg-equals-pg}, we see that $^\perp{\Pg {i+a}} =
{\Sg i}^\perp = \left ( \Sl i,
  \R\Gamma_{\geq i} \Db {\coh X} \right )$, and thus there is a semiorthogonal
decomposition
\begin{equation}
\label{so-decomp-grA-using-cohX}
\Db \grA = \left (\Pg {i+a}, \Sl i, \R\Gamma_{\geq i} \Db {\coh X}
\right ).
\end{equation}

The rest of the proof boils down to comparing the decompositions
\eqref{so-decomp-grA-using-dsg} and \eqref{so-decomp-grA-using-cohX},
depending on the sign of $a$.

Assume first that $a \geq 0$. Then $\Pg {i+a} \subseteq \Db \grAi$ by
definition, and $\Db \grAi =
{}^\perp \Sl i$ by \ref{sli_admiss}.
Hence the first two factors of \eqref{so-decomp-grA-using-cohX} are mutually orthogonal and we may swap them to get
\begin{equation}
  \label{so-decomp-when-a-positive}
  \Db \grA = \left (\Sl i,\Pg {i+a}, \R\Gamma_{\geq i} \Db {\coh
      X}\right ).
\end{equation}
Comparing with \eqref{so-decomp-grA-using-dsg} we see that
\begin{displaymath}
 \Db \grAi = \left (\Pg i, \Bi \right ) = \left (\Pg {i+a},
   \R\Gamma_{\geq i} \Db {\coh X} \right ).
\end{displaymath}
By \ref{except-objects-so-decomp-lemma} there
is a decomposition
\[ \Pg i = \left (\Pg {i+a}, A(-i -a +1), \ldots, A(-i + 1), A(-i) \right ) . \]
It follows there is an equality in $\Db \grAi$:
\begin{displaymath}
 \left ( A(-i -a + 1), \ldots, A(-i + 1), A(-i),  \Bi \right ) = \R\Gamma_{\geq i} \Db {\coh X}.
\end{displaymath}
Applying
$\twid{(-)}$ to both sides gives the semiorthogonal decomposition
\[ \Db {\coh X} = \left ( \cO(-i-a+1), \ldots, \cO(-i), \twid{\Bi}
\right ).\] 

Assume now that $a \leq 0$. In this case, $\Pg i \subseteq {\Sl
i}^{\perp}$, i.e.\ $\Hom {\Db \grA} {\Sl i}
{\Pg i} = 0$. To see this, it is enough to check that $\Hom {\Db
  \grA} {A_0(e)[m]} {A(f)} = 0$ for all $e > -i, f \leq -i$ and all
$m \in \Z$. But we have that
\begin{align*}
  \Hom {\Db \grA} {A_0(e)[m]} {A(f)} &\cong \Hom {\Db \grA} {A_0}
  {A}(f-e)[-m]\\
 &= \left ( H^{-m} \R\uHom \grA {A_0} {A} \right )_{f-e}.
\end{align*}
By the definition of Gorenstein, this is nonzero if and only if $-m =
n$ and $f - e = -a$, however $f - e < 0$ and $-a \geq 0$. Thus we may
switch the order of the first two factors of
\eqref{so-decomp-grA-using-dsg} and we have a semi-orthogonal
decomposition
\begin{equation} 
\label{so-decomp-grA-using-dsg-switched}
\Db \grA = (\Pg i, \Sl i, \Bi).
\end{equation}
We also have, substituting $i - a$ for $i$ in \eqref{so-decomp-grA-using-cohX},
\begin{displaymath}
  \label{so-decomp-grA-using-cohX-a}
\Db \grA = (\Pg {i}, \Sl {i-a}, \R\Gamma_{\geq i-a} \Db {\coh X}).
\end{displaymath}
This shows that $(\Sl i, \bb \Dsg \grA) = (\Sl {i-a}, \R\Gamma_{\geq
  i-a} \Db {\coh X})$. By \ref{except-objects-so-decomp-lemma} we have
$\Sl {i - a} = \left (\Sl i, A_0(-i), A_0(-i - 1),
\ldots, A_0(-i + a + 1)\right )$. Thus we have
\[ \Bi = \left (A_0(-i), A_0(-i - 1),
\ldots, A_0(-i + a + 1), \R\Gamma_{\geq
  i-a} \Db {\coh X} \right )\]
and applying the functor $p: \Db \grAi \to \Dsg \grA$ gives the desired decomposition.
\end{proof}

\section{Complete intersection rings and matrix factorizations}
In this section we apply the main theorem to relate the derived
category of a commutative complete intersection ring to the homotopy
category of graded matrix
factorizations over a ``generic hypersurface.''

\subsection{Graded matrix factorizations}
Let $S = \bigoplus_{i \geq 0} S_i$ be a commutative noetherian graded
ring and let $W \in S_d$, for some $d \geq 1$. A
\emph{graded matrix factorization} of $W$ is a pair of graded projective $S$-modules
$E_1, E_0$ and morphisms in $\grS$,
\[ e_1: E_1 \to E_0 \qquad e_0: E_0 \to E_1(d) \]
such that $e_0 e_1 = W\cdot 1_{E_1}$ and $e_1(d) e_0 = W \cdot
1_{E_0}$. A morphism $h$ between
$\bE = \emf d$ and $\mathbb{F} = \fmf d$ is a pair of
maps $h_1: E_1 \to F_1$ and $h_0: E_0 \to F_0$ making the obvious
diagrams commute. One defines a homotopy between two such maps
analogously to the case of a map of complexes. The category with
objects graded matrix factorizations of $W$ and morphisms homotopy
equivalence classes of morphisms of matrix factorizations is called
the \emph{homotopy category of matrix factorizations} and denoted
$\hgrmf$.

Now assume that $S_0$ is a regular commutative ring and $S$  is a
polynomial ring over $S_0$. Set $A = S/(W)$ and consider the
singularity category $\Dsg {\grA}$ as defined in \ref{defn_sing_cat}. The assignment that sends
$\bE = \emf d$ to the image of $\coker e_1$ in $\Dsg {\grA}$ induces a
functor
\begin{equation}
\label{gr_mf_sing_equiv} \coker: \hgrmf \to \Dsg {\grA}.
\end{equation}
It follows from work of
Eisenbud \cite{Ei80} and and Buchweitz \cite{Bu87}, but seems to have
first been written down by Orlov in \cite[\S 3]{MR2641200} that this
functor is an equivalence of categories.

\subsection{Generic hypersurface}
Let $R = Q/(\bf)$, where $Q$ is a commutative regular ring of finite
Krull dimension, and $\bf =
f_1, \ldots, f_c$ is a $Q$-regular sequence. Define $S = Q[T_1, \ldots,
T_c]$ to be the graded polynomial ring over $Q$ with $|T_i| = 1$. Let $W = f_1T_1 +
\ldots + f_c T_c \in S_1$ and set $A = S/(W)$.

Let $Y = \Proj A$ and note that there is a diagram
\begin{equation} \label{E621}
\xymatrix{\psc R = \Proj {\left (S \otimes_Q R \right )} \ar[r]^-{\RY} \ar[d]_{\pi}& Y \ar[r] &
  \Proj S = \psc Q \ar[d] \\
 \Spec R \ar[rr] & & \Spec Q} 
\end{equation}
where the vertical arrows are the canonical proper maps
and each horizontal arrow is a regular closed immersion and thus has
finite Tor dimension. In particular the map $\RY: \psc R \to Y$ is
a regular closed immersion of codimension $c-1$. Orlov used this
setup in \cite{MR2437083} to show that there is an equivalence between
the singularity categories of $R$ and $Y$. This equivalence was used in
\cite{BW11b} and \cite{1105.4698}.

\newcommand{\cR}{\mathcal{R}}
\begin{lem}
\label{desc_of_perp_of_R_in_Y}
The functor $\RY_* \pi^*: \Db R \to \Db {\coh Y}$ is fully
  faithful and has a right adjoint. Thus the image $\cR$ is a right admissible subcategory of $\Db {\coh Y}$ equivalent to $\Db R$. Moreover, 
the right orthogonal of $\cR$ is
\[ (\RY_* \pi^* \Db R)^\perp = \gend{\cO_Y(-c+2),\ldots, \cO_Y(-1), \cO_Y}.\]
\end{lem}

\begin{proof}
Orlov shows in \cite[2.2]{MR2437083} that the functor $\RY_* \pi^*: \Db R \to \Db {\coh Y}$ is fully
  faithful and has a right adjoint (the existence of a right adjoint
  to $\RY_*$ is one formulation of Grothendieck duality in this
  context). He also shows in \cite[2.10]{MR2437083} that the left
  orthogonal of the image is $\gend{ \cO_Y(1), \ldots, \cO_Y(c-1)}$; a
  slight reworking of this argument shows the right orthogonal is as claimed.
\end{proof}

\subsection{The equivalence}
We continue to assume that $R$ is a complete intersection of the form
$Q/(\bf)$, where $Q$ is a commutative regular ring of finite Krull dimension, and $\bf =
f_1, \ldots, f_c$ is a $Q$-regular sequence. Recall that $A = Q[T_1,
\ldots, T_c]/(f_1 T_1 + \ldots + f_c T_c) = S/(W)$. We wish to apply
Theorem \ref{main-thm} to this ring. We first must show $A$ is
Gorenstein. This holds by ``graded
local duality'' as in \cite[\S 3.4]{MR1251956}.
\begin{lem}
  There is an isomorphism in
  $\Db\grA$,
\[ \RuHom \grA {A_0} A \cong A_0[-n](c-1),\]
where $n = \dim A$. In
particular $A$ is a Gorenstein ring with $a$-invariant $c-1$.
\end{lem}


We now have:
\begin{thm}
 There is an equivalence
\[ \Psi: \Db R
\xra{\cong} \hgrmf \]
given by $\Psi =  q\left(\R \Gamma_{\geq 0}\right)\RY_*
\pi^*$, where $q$ is the composition $\Db \grAo \xra{p} \Dsg \grA \xra{\cong} \hgrmf$.
\end{thm}

\begin{proof}
 By Theorem \ref{main-thm} applied to $A$ with $i = 0$, we know that
 $\Db {\coh Y}$ has a semiorthogonal decomposition $\left ( \cO_Y(-c + 2), \ldots, \cO_Y,
\twid{\Bi}\right )$
where $\Bi$ is the
image of $\Dsg \grA$ under the fully faithful functor $\bb_i\colon \Dsg \grA
\to \Db \grAi$ described in \eqref{defn-bbi}. Thus $\twid{\Bi}^{\perp}
= \left ( \cO_Y(-c + 2), \ldots, \cO_Y \right )$, which by
Lemma~\ref{desc_of_perp_of_R_in_Y}  is also equal to $\cR^{\perp}$,
where $\cR$ is the image of $\Db R$ under $\RY_*
\pi^*$. Thus $\cR = \twid{\Bi}$ and applying $q \R \Gamma_{\geq 0}$ to
both sides we have an equivalence
\[ \Db R \xra{ q \R \Gamma_{\geq 0}\RY_*
\pi^*} \Dsg \grA.\]
Finally, the equivalence \eqref{gr_mf_sing_equiv} finishes the proof.
\end{proof}

\newcommand{\cE}{\mathcal{E}}
In \cite{BW11b}, it was shown that there is an equivalence
\[ \Dsg R \cong [MF(\psc Q, \cO(1), W)], \]
where $[MF(\psc Q, \cO(1), W)]$ is the homotopy category of
matrix factorizations of locally free sheaves on $\psc Q$. This
category has objects pairs of locally free sheaves $(\mathcal{E}_1, \mathcal{E}_0)$ on $\psc Q$ and
maps $e_1: \cE_1 \to \cE_0$ and $e_0: \cE_0 \to \cE_1(1)$ such that
composition is multiplication by $W$. Morphisms are defined analogously as in the affine case above, however
there is a further localization
at objects that are locally contractible. There is an obvious functor
${\widetilde{(-)}}: \hgrmf \to [MF(\psc Q, \cO(1), W)]$. This
equivalence fits into the following commutative diagram, where the
left hand arrow is the natural projection onto the singularity category.
\[ \xymatrix{ \Db R \ar[rr]_-\cong \ar[d] && \hgrmf \ar[d]^-{\widetilde{(-)}}\\
\Dsg R \ar[rr]_-\cong &&  [MF(\psc Q, \cO(1), W)].
}\]


\def\cprime{$'$}

\end{document}